\documentclass[12pt]{article}

\usepackage{amsfonts,amsmath,amssymb}
\usepackage{mathrsfs,mathtools,url}
\usepackage{latexsym}
\usepackage{tikz,color}
\usepackage[T1]{fontenc}
\usepackage{array}
\usepackage{xspace}
\usepackage{cite}

\newtheorem{theorem}{Theorem}[section]

\newtheorem{lemma}[theorem]{Lemma}
\newtheorem{proposition}[theorem]{Proposition}

\newtheorem{corollary}[theorem]{Corollary}

\newtheorem{problem}{Problem}
\newtheorem{claim}{Claim}

\DeclareMathOperator {\diam} {diam}

\DeclareMathOperator {\mut} {\mu_{\rm t}}
\DeclareMathOperator {\muo} {\mu_{\rm o}}
\DeclareMathOperator {\mud} {\mu_{\rm d}}
\DeclareMathOperator {\muti} {\mu_{\rm t}^{\rm i}}
\DeclareMathOperator {\muoi} {\mu_{\rm o}^{\rm i}}
\DeclareMathOperator {\mudi} {\mu_{\rm d}^{\rm i}}
\DeclareMathOperator {\mui} {\mu^{\rm i}}

\def\cp{\,\square\,}
\newcommand{\proof}{\noindent{\bf Proof.\ }}
\newcommand{\qed}{\hfill $\square$ \bigskip}
\newcommand{\smallqed}{{\tiny ($\Box$)}}

\title{Independent mutual-visibility sets and distance edge-critical graphs}

\author{Jing Tian$\/^{a}$, Csilla Bujt\'{a}s$^{b, c}$, Sandi Klav\v{z}ar$^{b, c, d}$ \\\\
$^{a}$ \small School of Science, Zhejiang University of Science and Technology, \\
\small Hangzhou, Zhejiang 310023, PR China\\
\small {\tt jingtian526@126.com \quad ORCID: 0000-0002-1578-4798}\\
$^{b}$ \small Faculty of Mathematics and Physics, University of Ljubljana, Slovenia\\
\small {\tt csilla.bujtas@fmf.uni-lj.si\quad ORCID: 0000-0002-0511-5291}\\
\small {\tt sandi.klavzar@fmf.uni-lj.si\quad ORCID: 0000-0002-1556-4744}\\
$^{c}$ \small Institute of Mathematics, Physics and Mechanics, Ljubljana, Slovenia \\
$^{d}$ \small Faculty of Natural Sciences and Mathematics, University of Maribor, Slovenia\\
}

\date{\today}

\begin{document}

\maketitle

\begin{abstract}
In this paper, connections between independent sets and the variety of mutual-visibility sets are studied. It is proved that every outer mutual-visibility set of a graph is independent if and only if the graph is distance edge-critical. Several constructions yielding distance edge-critical graphs are given. Graphs in which every independent set is a total mutual-visibility or a dual mutual-visibility set are characterized, as well as graphs in which every total mutual-visibility set is independent. Along the way the total mutual-visibility number of some graphs derived from fullerenes is determined. Graphs in which every independent set is a mutual-visibility set are discussed and characterized over diameter four graphs. It is proved that determining the maximum cardinality of an independent mutual-visibility set and deciding whether it equals the independence number of a graph are NP-hard problems, and the same is true for independent total, outer and dual mutual-visibility sets. 
\end{abstract}


\medskip\noindent
\textbf{AMS Math.\ Subj.\ Class.\ (2020)}: 05C12, 05C69, 05C76, 68Q17

\section{Introduction}

The mutual-visibility problem was in 2022 transferred from computer science to graph theory in Di Stefano's paper~\cite{distefano-2022}. The problem immediately attracted widespread interest, in part because it is closely related to several other important mathematical topics. Let us emphasize its connections to Tur\'an type problems~\cite{boruzanli-2024, bujtas-2025a, cicerone-2024b}, to the famous open Zarankiewicz problem~\cite{cicerone-2023}, and to the Bollob\'as-Wessel theorem~\cite{bollobas-1967, wessel-1966}, see~\cite{BresarYero-2024}. Moreover, to prove that there exist large mutual-visibility sets in hypercubes, Axenovich and Liu~\cite{axenovich-2024a+} applied a recent  breakthrough result on daisy-free hypergraphs due to Ellis, Ivan, and Leader~\cite{ellis-2024}. From among other papers on graph mutual-visibility we highlight~\cite{cicerone-2025, dettlaff-2026, korze-2024, korze-2025, kuziak-2023, roy-2025a, roy-2025b, tian-2024}. 

While investigating mutual-visibility in strong products~\cite{cicerone-2024a}, the need for the total mutual-visibility concept appeared. The whole variety was later (although published earlier) introduced in~\cite{CiDiDrHeKlYe-2023} as follows. If $X$ is a subset of vertices of a graph $G$, then $u,v\in V(G)$ are {\em $X$-visible}, if there exists a shortest $u,v$-path $P$ such that $V(P)\cap X \subseteq \{u,v\}$. If every two vertices from $X$ are $X$-visible, then $X$ is a \emph{mutual-visibility set} ({\em MV set} for short), and if any two vertices from $V(G)$ are $X$-visible, then $X$ is a \emph{total mutual-visibility set} ({\em total MV set}). Further, if any two vertices from $X$ and any two vertices from $\overline{X}$ are $X$-visible, then $X$ is a \emph{dual mutual-visibility set} ({\em dual MV set}), and if any two vertices from $X$ are $X$-visible, and any two vertices $x\in X$, $y\in \overline{X}$ are $X$-visible, then $X$ is an \emph{outer mutual-visibility set} ({\em outer MV set}). ($\overline{X}$ stands for $V(G)\setminus X$.) The cardinality of a largest MV set (resp.\ total/dual/outer MV set) is the \emph{mutual-visibility number} (resp.\ \emph{total/dual/outer mutual-visibility number}) $\mu(G)$ (resp.\ $\mut(G)$, $\mud(G)$, $\muo(G)$) of $G$. 

To derive general bounds for the mutual-visibility number in Cartesian products~\cite{cicerone-2023}, independent MV sets were introduced as those MV sets which induce edgeless graphs. Next, Kor\v{z}e and Vesel~\cite[Proposition 3.8]{korze-2025} proved that if $d\ge 1$, then $\mut(Q_d) = 2\alpha(Q_d')$, where $Q_d'$ is the dual cube of dimension $d$, and $\alpha(\cdot)$ stands for the independence number. The following result intrigued us further, where $g(G)$ denotes the girth of a graph $G$. 

\begin{proposition} {\rm \cite[Lemma 5.2]{cicerone-2024b}}
\label{prop:outer-independent}
If $G$ is a connected graph of order at least $3$ and with $g(G) \ge 5$, then every outer MV set of $G$ is independent.
\end{proposition}

In view of this proposition and other mentioned results, we set ourselves the task of systematically reviewing when all independent sets occur to be MV sets and the other way around. 

After some preparation in the next section, we consider outer mutual-visibility in Section~\ref{sec:outer}. We prove that every outer MV set of a graph is independent if and only if the graph is distance edge-critical. We also  provide several constructions yielding distance edge-critical graphs, and consider the graphs in which every independent set is an outer MV set. In Section~\ref{sec:total/dual} we prove that the so-called dense diamond graphs characterize the graphs in which every independent set is a total MV or a dual MV set for that matter. We also characterize graphs in which every total MV set is independent, and provide a partial description of graphs in which every dual MV set is independent. Along the way we determine the total mutual-visibility number of some graphs derived from fullerenes. In Section~\ref{sec:visibility} we first describe graphs in which every independent set is a MV set as the graphs which contain no independent, strong critical set. Then we give a polynomial-time characterization of such graphs among diameter four graphs, and use the strong product to construct graphs of an arbitrarily large diameter in which every independent set is a MV set. In Section~\ref{sec:NP-c} we prove that determining the maximum cardinality of an independent mutual-visibility set and deciding whether it equals the independence number of a graph are NP-hard problems over the class of graphs of diameter $d$ for every fixed $d\ge 4$, and that the same is true for independent total, outer and dual mutual-visibility sets. We conclude the paper by proposing several open problems. 

\section{Preliminaries}

Let $n(G)$ and $m(G)$ respectively denote the order and the size of a graph $G$. For vertices $u$ and $v$ of $G$, the length of a shortest $u,v$-path is the {\em distance} between $u$ and $v$ and denoted by $d_G(u,v)$. The {\em interval} $I_G[u,v]$ is the set of vertices which lie on shortest $u,v$-paths in $G$. The {\em diameter} $\diam(G)$ of $G$ is the length of a longest shortest path in $G$. A subgraph $H$ of $G$ is {\em convex}, if for every two vertices of $H$, all shortest $u,v$-paths belong to $H$. 

A set $S\subset V(G)$ of $G$ is {\em distance $\{u, v\}$-critical}, if $u, v\in V(G)\setminus S$ and $d_{G-S}(u,v) > d_G(u,v)$. In the case when $S = \{x\}$ we say that $x$ is a {\em distance critical vertex}. Similarly, an edge $xy$ is {\em distance critical}, if there exist vertices $u, v\in V(G)$, where $\{u,v\} \ne \{x,y\}$, such that $d_{G-xy}(u,v) > d_G(u,v)$. Further, $G$ is {\em distance vertex-critical} if every vertex of $G$ is distance critical and is  {\em distance edge-critical} if every edge of $G$ is distance critical. Since this will be the only type of criticality considered in this paper, we will simplify the terminology to {\em critical set}, {\em critical vertex}, {\em critical edge}, {\em vertex-critical graph}, and {\em edge-critical graph}. 

The first statement of the following lemma can be found in~\cite[Theorem 1]{erdos-1980}. In this paper by Erd\H{o}s and Howorka, where vertex-critical graphs were introduced, the focus was on the largest number of edges such a graph can contain. As far as we know, nothing followed in this direction for almost half of a century, until these graphs were systematically researched for the first time in the recent paper~\cite{cooper-2025}.

\begin{lemma}
\label{lem:critical}
If $G$ is a connected graph, then the following hold.
\begin{enumerate}
    \item[(i)] $G$ is vertex-critical if and only if every vertex of $G$ is the middle vertex of a convex $P_3$~{\rm \cite{erdos-1980}}. 
    \item[(ii)] $G$ is edge-critical if and only if every edge of $G$ lies on a convex $P_3$.
\end{enumerate}
\end{lemma}

\proof
We only need to prove (ii). Assume first that every edge of $G$ lies on a convex $P_3$ and let $e=uv$ be an arbitrary edge. Let $uvw$ be a convex $P_3$ containing $e$. Then $d_{G-e}(u,w) \ge 3$, meaning that $e$ is a critical edge. Since this conclusion holds for every edge, $G$ is edge-critical. 

Assume second that $G$ is an edge-critical graph and let $e=uv$ be an arbitrary edge of $G$. Let $x$ and $y$ be vertices such that $d_{G-e}(x,y) > d_G(x,y)$, where $\{x,y\} \ne \{u,v\}$. This means that every shortest $x,y$-path in $G$ contains $e$. Consider an arbitrary shortest $x,y$-path $P$. Then $P$ contains a subpath $P' = uvw$ or $wuv$. Since every shortest $x,y$-path contains $e=uv$, the subpath $P'$ is a convex $P_3$ containing $e$.
\qed

The sets of vertex-critical and edge-critical graphs are incomparable. For instance, the graph obtained from $C_4$ by attaching a pendant vertex to two of its non-adjacent vertices is edge-critical but not vertex-critical. On the other hand, the Wagner graph (as shown in Fig.~\ref{fig:C8}) is vertex-critical but not edge-critical. 

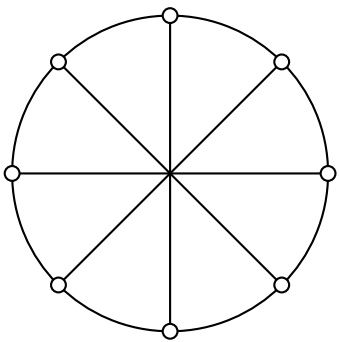
\begin{figure}[ht!]
\begin{center}
\begin{tikzpicture}[scale=0.7,style=thick,x=1cm,y=1cm]
\def\vr{4pt}
\def\radius{3cm}
    
\begin{scope}[xshift=0cm, yshift=0cm] 
\coordinate(x1) at (0,3);
\coordinate(x2) at (2.12,2.12);
\coordinate(x3) at (3,0);
\coordinate(x4) at (2.12,-2.12);
\coordinate(x5) at (0,-3);
\coordinate(x6) at (-2.12,-2.12);
\coordinate(x7) at (-3,0);
\coordinate(x8) at (-2.12,2.12);
\draw[] (0,0) circle[radius=\radius];		
\draw (x1) -- (x5);
\draw (x2) -- (x6);
\draw (x3) -- (x7);
\draw (x4) -- (x8);

\draw(x1)[fill=white] circle(\vr);
\draw(x2)[fill=white] circle(\vr);
\draw(x3)[fill=white] circle(\vr);
\draw(x4)[fill=white] circle(\vr);
\draw(x5)[fill=white] circle(\vr);
\draw(x6)[fill=white] circle(\vr);
\draw(x7)[fill=white] circle(\vr);
\draw(x8)[fill=white] circle(\vr);
\end{scope}

\end{tikzpicture}
\caption{Wagner graph $M_8$}
\label{fig:C8}
\end{center}
\end{figure}

Let $G$ and $H$ be disjoint graphs. Then the {\em join} $G\vee H$ is the graph obtained from the disjoint union of $G$ and $H$ by adding the edge $gh$ for every $g\in V(G)$ and every $h\in V(H)$. Further, when we say that $X$ is a {\em join graph}, we mean that $X$ can be represented as $G\vee H$. The Cartesian product $G\cp H$ has the vertex set $V(G)\times V(H)$, vertices $(g,h)$ and $(g',h')$ being adjacent if either $g=g'$ and $hh'\in E(H)$, or $h=h'$ and $gg'\in E(G)$. The {\em strong product}  $G\boxtimes H$ has the vertex set $V(G\boxtimes H) = V(G) \times V(H)$ and $E(G\boxtimes H) = E(G\cp H)\cup\{\{(g,h),(g',h')\}: gg'\in E(G)\ {\rm and}\ hh'\in E(H)\}$. The subgraph of $G\cp H$ (resp., $G\boxtimes H$) induced by the vertices $(g,h)$, $g\in V(G)$, is called a {\em $G$-layer} of $G\cp H$ (resp., $G\boxtimes H$) and is denoted by $G^h$. For a vertex $g\in G$, the $H$-layer $^gH$ is defined analogously. Note that $G^h \cong G$ and $^gH \cong H$.

To conclude the preliminaries, we recall the following result needed later on. 

\begin{theorem} {\rm \cite[Theorem 5.2]{Bujtas-2025b}}
\label{thm:distance two}
If $G$ is a connected graph and $X\subseteq V(G)$, then the following statements are equivalent. 
\begin{enumerate}
    \item[(i)] $X$ is a total MV set of $G$.
    \item[(ii)] Any two vertices $u$ and $v$ of $G$ with $d_G(u,v) = 2$ are $X$-visible.  
    \item[(iii)] Any two vertices $u$ and $v$ of $G$ with $d_G(u,v) = 2$ satisfy $N_G(u) \cap N_G(v) \not\subseteq X$.
\end{enumerate}
\end{theorem}

\section{Outer mutual-visibility and independence}
\label{sec:outer}

As already mentioned, Proposition~\ref{prop:outer-independent} was one of our main motivations for the present investigation. We now extend this results.  

\begin{theorem} 
\label{thm:outer-independent}
If $G$ is a connected graph of order at least three, then the following statements are equivalent.
\begin{itemize}
\item[$(i)$] Every outer MV set of $G$ is independent.
\item[$(ii)$] $G$ is an edge-critical graph.
\end{itemize}
\end{theorem}

\proof 
(i) $\Rightarrow$ (ii): Let $uv$ be an edge of $G$. By (i), the set $X=\{u,v\}$ is not an outer MV set in $G$. As $u$ and $v$ are $X$-visible, there exists a vertex $z \in V(G) \setminus X$ such that a vertex from $X$ and $z$ are not $X$-visible. Without loss of generality, we may suppose that every shortest $u,z$-path contains $v$. Then, if $uvv'\dots z$ is a shortest $u,z$-path, then $uvv'$ is a convex $P_3$. Since this happens for every edge of $G$, Lemma~\ref{lem:critical}(ii) implies that $G$ is an edge-critical graph. 

(ii) $\Rightarrow$ (i):
Let $X$ be an outer MV set of $G$ and suppose on the contrary that $X$ contains adjacent vertices $u$ and $v$. Since $G$ is an edge-critical graph, by Lemma~\ref{lem:critical}(ii) there exists a vertex $w$ such that, without loss of generality, $uvw$ is a convex $P_3$. Then $u$ and $w$ are not $X$-visible, which is a contradiction no matter whether $w\in X$ or $w\notin X$.
\qed

Restricting to graphs of diameter two, Theorem~\ref{thm:outer-independent} further strengthens as follows. 

\begin{proposition}
\label{prop:outer-inde in diam 2}
If $G$ is a graph with $\diam(G) = 2$, then every independent set is an outer MV set. In addition, if $G$ is edge-critical, then outer MV sets coincide with independent sets and $\muo(G) = \alpha(G)$.
\end{proposition}

\proof
Let $S$ be an independent set of a graph $G$, where $\diam(G) = 2$. Then every two vertices of $S$ are clearly $S$-visible. Moreover, if $x\in S$ and $y\notin S$, then either $xy\in E(G)$ or $d_G(x,y) = 2$. In any case, since $S$ is independent, $x$ and $y$ are also $S$-visible. 

Assume next that $G$ is also edge-critical. Then by Theorem~\ref{thm:outer-independent}, every outer MV set is independent and we are done. 
\qed

Note that the Wagner graph $M_8$ from Fig.~\ref{fig:C8} is an example of a diameter two graph which is not edge-critical and hence not every outer MV set of $M_8$ is independent.  

\subsection{Edge-critical graphs}

In view of Theorem~\ref{thm:outer-independent} and Proposition~\ref{prop:outer-inde in diam 2}, edge-critical graphs deserve a closer attention. 

Examples of edge-critical graphs include graphs $G$ with $g(G)\ge 5$ and trees of order at least three. In addition, if $G$ is an edge-critical graph, then every graph obtained from $G$ by subdividing every edge of $G$ arbitrary number of times (possibly zero) is also an edge-critical graph. Among the Cartesian product graphs, edge-critical graphs are characterized as follows. 

\begin{proposition} 
\label{prop:cp}
If $G$ and $H$ are graphs, then $G\cp H$ is edge-critical if and only if both $G$ and $H$ are edge-critical. 
\end{proposition}

\proof
Assume first that $G\cp H$ is edge-critical. Let $e = (g,h)(g',h)\in E(G\cp H)$. Then by Lemma~\ref{lem:critical}(ii), $e$ lies in a convex $P_3$ of $G\cp H$. By the structure of $G\cp H$, this convex $P_3$ must lie in the layer $G^h$. This in turn implies that $gg'\in E(G)$ lies in a convex $P_3$ in $G$. So $G$ is an edge-critical graph, and by the commutativity of the Cartesian product, $H$ is an edge-critical graph as well. 

Conversely, assume that $G$ and $H$ are edge-critical graphs. Then using the fact that the layers of $G\cp H$ are convex subgraphs, and every edge of $G\cp H$ lies in some layer, we can conclude that $G\cp H$ is an edge-critical graph. 
\qed

Also among graphs of diameter two, the variety of edge-critical graphs remains large as illustrated by the following three families of such graphs.

\begin{itemize}
\item $K_{1,n}$, $n\ge 2$.
\item Let $\mathcal{G}$ be a family of graphs introduced in~\cite{henning-2015} as follows. Let $P$ be the Petersen graph, and let $G_7$ be the graph obtained from the cycle $C_3$ by adding a pendant edge to each vertex of the cycle and then adding a new vertex and joining it to the three degree-one vertices, see Fig.~\ref{fig:three graphs}. Then the family $\mathcal{G}$ 
\begin{enumerate}
\item[(i)] contains $C_5$, $G_7$, and $P$; and 
\item[(ii)] is closed under degree-two vertex duplication (see Fig.~\ref{fig:three graphs} again).
\end{enumerate}
It is straightforward to verify that every $G\in \mathcal{G}$ is an edge-critical graph with $\diam(G) = 2$. For instance, in $G_7$ we infer that every edge lies in a convex $P_3$ such that its middle vertex is not of degree $2$. This implies that after a degree-two vertex duplication is performed, the constructed graph remains edge-critical. 
\item The graph $H_{n,m}$, $n\ge 3$, $m\ge 1$, is defined as follows. Let $A$ and $B$ be the two bipartition sets of $K_{n,m}$, where $|A| = n$ and $|B| = m$. Then $H_{n,m}$ is obtained from the disjoint union of $K_{n,m}$ and $K_n$ by adding a matching between the vertices of $K_n$ and the vertices of $A$. See Fig.~\ref{fig:three graphs} where $H_{4,2}$ is shown and note that $H_{3,1} = G_7$. It is again straightforward to verify that every $H_{n,m}$ is an edge-critical graph of diameter two. Moreover, every graph obtained from $H_{n,1}$ by a degree-two vertex duplication is edge-critical and of diameter two.
\end{itemize}

\begin{figure}[ht!]
\begin{center}
\begin{tikzpicture}[scale=1,style=thick,x=0.8cm,y=0.8cm]
\def\vr{2pt}

\begin{scope}[xshift=0cm, yshift=0cm] 
\coordinate(x1) at (0,0);
\coordinate(x2) at (0,1);
\coordinate(x3) at (-1,1);
\coordinate(x4) at (1,1);
\coordinate(x5) at (0,2);
\coordinate(x6) at (-1,2.5);
\coordinate(x7) at (1,2.5);

\draw (x1) -- (x2)--(x5)--(x7)--(x4)--(x1);
\draw (x1) -- (x3)--(x6)--(x5);
\draw (x6) -- (x7);

\draw(x1)[fill=white] circle(\vr);
\draw(x2)[fill=white] circle(\vr);
\draw(x3)[fill=white] circle(\vr);
\draw(x4)[fill=white] circle(\vr);
\draw(x5)[fill=white] circle(\vr);
\draw(x6)[fill=white] circle(\vr);
\draw(x7)[fill=white] circle(\vr);

\node at (0.2,-1) {$G_7 \cong H_{3,1}$ };    
\end{scope}

\begin{scope}[xshift=4cm, yshift=0cm] 
\coordinate(x1) at (0,3);
\coordinate(x2) at (-2,0);
\coordinate(x3) at (2,0);
\coordinate(x4) at (0,1.3);
\coordinate(x5) at (-0.25,2);
\coordinate(x6) at (0.25,2);
\coordinate(x7) at (-1,1);
\coordinate(x8) at (-0.8,0.6);
\coordinate(x9) at (0.65,0.5);
\coordinate(x10) at (0.9,1);

\draw (x1) -- (x2)--(x3)--(x1);
\draw (x1)--(x5)--(x4)--(x6)--(x1);
\draw (x2) -- (x7)--(x4)--(x8)--(x2);
\draw (x3) -- (x9)--(x4)--(x10)--(x3);

\draw(x1)[fill=white] circle(\vr);
\draw(x2)[fill=white] circle(\vr);
\draw(x3)[fill=white] circle(\vr);
\draw(x4)[fill=white] circle(\vr);
\draw(x5)[fill=white] circle(\vr);
\draw(x6)[fill=white] circle(\vr);
\draw(x7)[fill=white] circle(\vr);
\draw(x8)[fill=white] circle(\vr);
\draw(x9)[fill=white] circle(\vr);
\draw(x10)[fill=white] circle(\vr);

\end{scope}

\begin{scope}[xshift=7.5cm, yshift=0cm] 
\coordinate(x1) at (0,3);
\coordinate(x2) at (1,3);
\coordinate(x3) at (2,3);
\coordinate(x4) at (3,3);
\coordinate(x5) at (0,1.5);
\coordinate(x6) at (1,1.5);
\coordinate(x7) at (2,1.5);
\coordinate(x8) at (3,1.5);
\coordinate(x9) at (1,0);
\coordinate(x10) at (2.5,0);

\draw (x1) -- (x2)--(x3)--(x4);
\draw (x1)--(x5)--(x9)--(x6)--(x2);
\draw (x9)--(x7)--(x3);
\draw (x9)--(x8)--(x4);
\draw (x10)--(x6);
\draw (x10)--(x5);
\draw (x10)--(x7);
\draw (x10)--(x8);
\draw (x1) .. controls (1,3.3).. (x3);
\draw (x2) .. controls (2,3.3).. (x4);
\draw (x1) .. controls (1,3.5) and (2,3.5).. (x4);

\draw(x1)[fill=white] circle(\vr);
\draw(x2)[fill=white] circle(\vr);
\draw(x3)[fill=white] circle(\vr);
\draw(x4)[fill=white] circle(\vr);
\draw(x5)[fill=white] circle(\vr);
\draw(x6)[fill=white] circle(\vr);
\draw(x7)[fill=white] circle(\vr);
\draw(x8)[fill=white] circle(\vr);
\draw(x9)[fill=white] circle(\vr);
\draw(x10)[fill=white] circle(\vr);

\node at (1.75,-1) {$H_{4,2}$ };    
\end{scope}

\end{tikzpicture}
\caption{From left to right: $G_7 \cong H_{3,1}$; the graph obtained from $G_7$ by duplicating once each of its degree two vertices; $H_{4,2}$}
\label{fig:three graphs}
\end{center}
\end{figure}
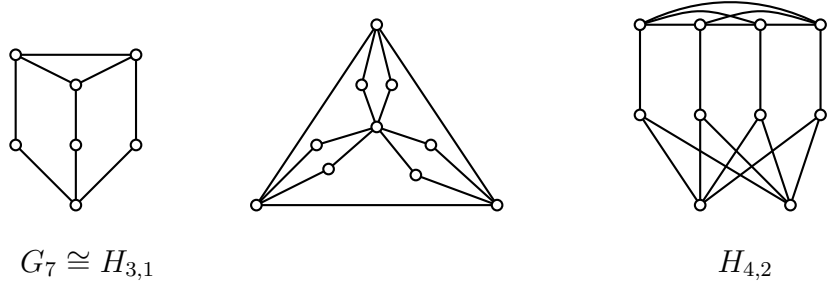

If $G$ is  an edge-critical graph with $\diam(G) = 2$ and has a leaf, then we can infer that $G$ must be a star. For triangle-free edge-critical graphs $G$ with $\diam(G) = 2$ and $\delta(G) \ge 2$ we have the following structural result. 

\begin{proposition}
\label{prop:triangle-free}
If $G$ is a triangle-free, edge-critical graph with $\diam(G) = 2$ and $\delta(G) \ge 2$, then every edge of $G$ lies in an induced $5$-cycle.
\end{proposition}

\proof
Let $e=xy\in E(G)$. Since $G$ is edge-critical, by Lemma~\ref{lem:critical}(ii) we may without loss of generality assume that there exists a neighbor $x'$ of $x$ such that the vertices $x', x, y$ induce a convex $P_3$. Since $\delta(G)\ge 2$, the vertex $y$ has a neighbor $y'\ne x$. Since $G$ is triangle-free, $y'\ne x'$. Thus the vertices $x', x, y, y'$ are pairwise different and lie on a $P_4$. Moreover, this $P_4$ is induced since the presence of the edge $x'y'$ would imply that $x', x, y$ would not induce a convex $P_3$. It follows, having in mind that $\diam(G) = 2$, that there exists a new vertex $z$ such that $x'z, zy'\in E(G)$. The triangle-freeness of $G$ also gives $zx\notin E(G)$ and $zy\notin E(G)$. We can conclude that $xx'zy'yx$ is an induced $C_5$ containing the edge $e$.
\qed

The graph obtained from $C_5$ by duplicating one of its vertices is a graph that fulfills all the assumptions of Proposition~\ref{prop:triangle-free}, yet it contains no convex $C_5$. Therefore, we cannot strengthen Proposition~\ref{prop:triangle-free} by stating that every edge of $G$ lies in a convex $C_5$.     

The following example shows that triangle-freeness cannot be omitted in Proposition~\ref{prop:triangle-free}. We take a $6$-cycle $v_1\dots v_6v_1$ and add the edges $v_1v_3$ and $v_1v_5$. For the graph $G$ obtained this way, $\diam(G)=2$ and $\delta(G)\geq2$. We can also identify convex $P_3$-subgraphs in $G$ that together cover the entire edge set. Since $N[v_6] \cap N[v_2]= \{v_1\}$, the path $v_6v_1v_2$ is convex. Similarly, $v_2v_3v_4$, $v_4v_5v_6$, $v_5v_1v_2$, and $v_3v_1v_6$ are all convex $P_3$-subgraphs. Therefore $G$ is edge-critical. On the other hand, there is no induced $5$-cycle in $G$.

\section{Total/dual mutual-visibility and independence}
\label{sec:total/dual}

In this section we consider connections between independent sets, and total MV sets and dual MV sets. 

The graph $K_4-e$ is known as the {\em diamond} graph. We say that a connected graph $G$ is a {\em dense diamond graph}, if any two vertices $u,v\in V(G)$ with $d_G(u,v) = 2$ lie on a common diamond. 

\begin{theorem}
\label{thm:dense-diamond-independent}
If $G$ is a connected graph, then the following statements are equivalent. 
\begin{enumerate}
\item[(i)] Every independent set of $G$ is a total MV set. 
\item[(ii)] Every independent set of $G$ is a dual MV set. 
\item[(iii)] $G$ is a dense diamond graph. 
\end{enumerate}
\end{theorem}

\proof
(i) $\Rightarrow$ (ii): This holds since total MV sets are dual MV sets.

(ii) $\Rightarrow$ (iii): Let $x,y\in V(G)$ be vertices with $d_G(x,y) = 2$. Then $W= N_G(x)\cap N_G(y)$ is not empty. If there exist adjacent vertices $a,b\in W$, then $x,y,a,b$ induce a diamond. Otherwise, $W$ induces an independent set. By the condition (ii), $W$ must be a dual MV set. But this is not true because $x$ and $y$ are not $W$-visible. 

(iii) $\Rightarrow$ (i): Let $G$ be a dense diamond graph and let $X$ be an independent set of $G$. We need to show that $X$ is a total MV set. By Theorem~\ref{thm:distance two} it suffices to demonstrate that if $x,y\in V(G)$ with $d_G(x,y) = 2$, then $x$ and $y$ are $X$-visible. By (iii), there exists a subgraph $K_4-e$ which contains $x$ and $y$, let $a$ and $b$ be the common neighbors of $x$ and $y$ in this diamond. As $ab\in E(G)$ we have $|X\cap \{a,b\}| \le 1$. It follows that $x$ and $y$ are $X$-visible. 
\qed

Consider the simple example of $C_4$ to see that the equivalence in Theorem~\ref{thm:dense-diamond-independent} does not remain true if we require that every independent set is a MV set or that it is an outer MV set.  
\medskip

We can also characterize graphs with the reverse property for total MV sets. A {\em convex diamond} $v_1v_2v_3v_4$ is a diamond where $v_1$ and $v_3$ are nonadjacent vertices and $N(v_1) \cap N(v_3) = \{v_2,v_4\}.$ With the same notation, $v_2v_4$ is the \emph{middle edge of the diamond}.

\begin{theorem}
\label{thm:total-only-independent}
The following statements are equivalent for every connected graph $G$. 
\begin{enumerate}
\item[(i)] Every total MV set of $G$ is independent.
\item[(ii)] For every edge $uv \in E(G)$ one of the following properties is true:
  \begin{itemize}
   \item[(a)] $u$ or $v$ is the middle vertex of a convex $P_3$; 
   \item[(b)] $uv$ is the middle edge of a convex diamond.
   \end{itemize}
\end{enumerate}
\end{theorem}

\proof
(i) $\Rightarrow$ (ii): Assume that every total MV set of $G$ is independent and consider an arbitrary edge $uv \in E(G)$. If $u$ or $v$ is the middle vertex of a convex $P_3$, then there is nothing to prove, hence assume that neither $u$ nor $v$ is the middle vertex of a convex $P_3$. We need to prove that $uv$ is the middle edge of a convex diamond.

As we have assumed that total MV sets are independent, $X = \{u,v\}$ is not a total MV set. Hence there exist two vertices $x$ and $y$ which are not $X$-visible. Consider the following two cases. 

\medskip\noindent
{\bf Case 1}: $\{x,y\} \cap \{u,v\} \ne \emptyset$.\\
Assume without loss of generality that $x=u$. As $u$ and $v$ are $X$-visible, we then have $y\ne v$. Consider an arbitrary shortest $y,x$-path $Q$. Since $x=u$ and $y$ are not $X$-visible, $v$ lies on $Q$. Let $y'$ be the neighbor of $v$ on $Q$ different from $x$; it is possible that $y'=y$. As $Q$ is a shortest path, we clearly have $y'x\notin E(G)$. Therefore, since $v$ is not the middle vertex of a convex $P_3$, there exists a path $y'v'x$, where $v'\ne v$. But this means that $y$ and $x$ are $X$-visible, a contradiction. Hence Case 1 never happens. 

\medskip\noindent
{\bf Case 2}: $\{x,y\} \cap \{u,v\} = \emptyset$.\\
Let $Q$ be an arbitrary shortest $x,y$-path. In the first subcase suppose that $Q$ contains the edge $uv$, and assume without loss of generality that $d_G(x,u) < d_G(x,v)$. Let $x'$ be the neighbor of $u$ on $Q$ different from $v$, and let $y'$ be the neighbor of $v$ on $Q$ different from $u$. Since $u$ is not the middle vertex of a convex $P_3$, there exists a vertex $u'\ne u$ such that $x'u', u'v\in E(G)$. Since $Q$ is a shortest path, $u'y'\notin E(G)$. Similarly, because $v$ is not the middle vertex of a convex $P_3$, considering the subpath $u'vy'$ we infer that there exists a vertex $v'\ne v$ such that $u'v', v'y'\in E(G)$. But then replacing the subpath $x'uvy'$ of $Q$ by the subpath $x'u'v'y'$ yields that $x$ and $y$ are $X$-visible. Hence, the first subcase cannot happen, that is, $Q$ contains exactly one vertex from $\{u,v\}$, say $u$. Let $x'$ and $y'$ be the neighbors of $u$ on $Q$, such that $ d_G(x,x') < d_G(x,u) < d_G(x,y')$. Now, as $u$ is not the middle vertex of a convex $P_3$, there exists a path of length $2$ on vertices $x', u', y'$, where $u'\ne u$. If $u'\ne v$, then $x$ and $y$ are $X$-visible, a contradiction. Hence, it is necessary that $u'=v$. We can conclude that the vertices $x',u,v,y'$ induce a convex diamond as required.     

\smallskip
(ii) $\Rightarrow$ (i): Assume now that every edge  of $G$ satisfies (a) or (b). Let $X$ be an arbitrary total MV set of $G$ and suppose that there are vertices $x,y\in X$ such that $xy\in E(G)$. Then $xy$ satisfies (a) or (b). In the first subcase suppose that $xy$ satisfies (a). Let, without loss of generality, $xyz$ be a convex $P_3$. But then $x$ and $z$ are not $X$-visible. In the second subcase suppose that $xy$ satisfies (b). Consider a convex diamond containing $xy$ as the middle edge, and let $x'$ and $y'$ be the other two vertices of it. But now $x'$ and $y'$ are not $X$-visible. We can conclude that no two vertices of $X$ can be adjacent. 
\qed

Clearly, every graph with $\mut(G) = 0$ fulfills (i) in Theorem~\ref{thm:total-only-independent}, but in such graphs no non-trivial independent set is a total MV set. For examples of such graphs see~\cite{tian-2024}. 

For graphs that apply to Theorem~\ref{thm:total-only-independent} and contain non-trivial total MV sets, consider a fullerene graph $F$, that is, a cubic plane graph in which all faces are pentagons or hexagons. As is well known, $F$ contains exactly $12$ pentagonal faces. Let now $\widehat{F}$ be the graph obtained from $F$ by inserting one vertex, called a {\em central vertex} of $\widehat{F}$, into each of the faces of $F$ and connect it to all the vertices of the face. Let $x$ be an arbitrary vertex of $F$. Then consider a neighbor $y$ of $x$ on a face of $F$, and the central vertex of another face containing $x$ but not $y$, to see that $x$ is the central vertex of a convex $P_3$ in $\widehat{F}$. Hence no total MV set of $\widehat{F}$ contains such a vertex. The same conclusion can be deduced for each central vertex which lies inside a hexagonal face. However, there is no such restriction on the central vertices of pentagons. Moreover, applying Theorem~\ref{thm:distance two} we can deduce that these vertices form a total MV set. We have thus deduced the following result that could be of independent interest. 

\begin{proposition}
\label{prop:fullerenes}
If $F$ is a fullerene graph, then $\mut(\widehat{F}) = 12$. 
\end{proposition}

If we replace total mutual-visibility with dual mutual-visibility in Theorem~\ref{thm:total-only-independent}, the equivalence does not hold. However, since every total MV set is a dual MV set, one direction of Theorem~\ref{thm:total-only-independent} remains true.

\begin{corollary}
    \label{cor:dual-only-independent}
  If every dual MV set of $G$ is independent, at least one of the following statements is true for every edge $uv \in E(G)$.
  \begin{itemize}
   \item[(a)] $u$ or $v$ is the middle vertex of a convex $P_3$; 
   \item[(b)] $uv$ is the middle edge of a convex diamond.
   \end{itemize}  
\end{corollary}

\subsection{Dense diamond graphs}
\label{sec:dense-diamond}

In view of the above results, it is worth having a closer look at dense diamond graphs. We first demonstrate that the family of dense diamond graphs is large. 

\begin{proposition}
\label{prop:joins-diamond-dense}
If $G$ and $H$ are graphs, then $G\vee H$ is a dense diamond graph if and only if one of the following conditions holds. 
\begin{enumerate}
\item[(i)] $m(G)\ge 1$, $m(H)\ge 1$.
\item[(ii)] $m(G) = 0$, $\diam(H)\le 2$, and $m(H)\ge 1$.
\item[(iii)] $m(H) = 0$, $\diam(G)\le 2$, and $m(G) \ge 1$.
\end{enumerate}
\end{proposition}

\proof
Set $X = G\vee H$. 

Consider arbitrary vertices $x,y\in V(X)$ with $d_X(x,y) = 2$. Assume first that $m(G)\ge 1$ and $m(H)\ge 1$, and let $gg'\in E(G)$ and $hh'\in E(H)$. If $x,y\in V(G)$, then $x,y,h,h'$ lie in a diamond, and if $x,y\in V(H)$, then $x,y,g,g'$ lie in a diamond. Assume second that $m(G) = 0$, $\diam(H)\le 2$, and $m(H)\ge 1$. If $x,y\in V(G)$, then consider an arbitrary edge $hh'$ of $H$ to see that $x,y,h,h'$ lie in a diamond. If $x,y\in V(H)$, then let $xzy$ be a shortest $x,y$-path in $H$. Then $x,z,y,g$ lie in a diamond, where $g$ is an arbitrary vertex of $G$. The case $m(H) = 0$, $\diam(G)\le 2$, $m(G) \ge 1$ is symmetric to the last case. 

To prove the converse, suppose that none of (i), (ii), and (iii) holds. If $m(G) = m(H) = 0$, then there is no diamond in $G\vee H$. By the symmetry, if suffices to consider the case when $m(G) = 0$ and $\diam(H)\ge 3$. If $H$ is not connected (that is, $\diam(H) = \infty$), then vertices $h$ and $h'$ which lie in different components of $H$ do not lie in a diamond in $X$. Let next $H$ be connected and let $h$ and $h'$ be vertices of $H$ with $d_H(h,h') = 3$. Then $h$ and $h'$ do not lie in a diamond of $X$ because all their common neighbors lie in $V(G)$ and are therefore independent. In summary, if none of (i), (ii), and (iii) holds, then $X$ is not a dense diamond graph. 
\qed

Another family of dense diamond graphs is the family of strong products $G\boxtimes K_n$, $n\ge 2$, where $G$ is an arbitrary connected graph. Recall that $\diam(X\boxtimes Y) = \max\{ \diam(X), \diam(Y)\}$, cf.~\cite[Proposition 5.4]{hik-2011}. Hence the graphs $G\boxtimes K_n$ not only provide dense diamond graphs of arbitrary diameter, but also diameter two graphs which are not joins. The latter fact follows from the following. 

\begin{proposition}
\label{prop:strong-join}
If $G$ is a non-trivial graph and $n\ge 2$, then $G\boxtimes K_n$ is a join graph if and only if $G$ is a join graph. 
\end{proposition}

\proof
Assume $G = G_1\vee G_2$. Then $G\boxtimes K_n = (G_1\boxtimes K_n) \vee (G_2\boxtimes K_n)$. 

Conversely, assume that $G\boxtimes K_n = H_1\vee H_2$. Then there exist vertices $g_1, g_2\in V(G)$, $g_1\ne g_2$, such that the layer $^{g_1}K_n$ contains a vertex of $H_1$ and the layer $^{g_2}K_n$ contains a vertex of $H_2$. We now partition $V(G)$ into $X_1$ and $X_2$ as follows. First add $g_1$ to $X_1$ and $g_2$ to $X_2$. Further, for any other vertex $g\in V(G)$, we add $g$ to $X_1$ if there exists $y\in V(K_n)$ such that $(g,y)\in V(H_1)$; otherwise we put $g$ to $X_2$. We now observe that if $(g,y)\in V(H_1)$ and $(g',y')\in V(H_2)$, where $g\ne g'$, then $gg'\in E(G)$. We can conclude that $G = G[X_1]\vee G[X_2]$.
\qed

\section{Mutual-visibility and independence}
\label{sec:visibility}

Clearly, if a graph $G$ has at least one edge, then not every MV set is independent. In this section we are interested in graphs in which every independent set is a MV set. Note first that this property holds for all graphs $G$ with $\diam(G)\le 3$, see~\cite[Lemma 2.1]{cicerone-2023}. 

Let $u$ and $v$ be two vertices of a graph $G$. Recall that $S\subset V(G)$ of $G$ is {\em $\{u, v\}$-critical}, if $d_{G-S}(u,v) > d_G(u,v)$ and $u, v\notin S$. We further say that $S$ is a {\em strong $\{u, v\}$-critical set} if vertices from $S$ are adjacent to none of $u$ and $v$.  If $S$ is (strong) $\{x,y\}$-critical for some vertices $x$ and $y$ we simply say that $S$ is a (strong) critical set. Using these concepts we can characterize graphs in which every independent set is a MV set in the following way. 

\begin{theorem}
\label{thm:indep-is-MV}
For a  graph $G$, the following statements are equivalent.
\begin{enumerate}
\item [(i)] Every independent set of $G$ is a MV set.
\item[(ii)] $G$ contains no independent, strong critical set. 
\end{enumerate}
\end{theorem}

\proof
Assume first that every independent set of $G$ is a MV set. Suppose for a contradiction that $S$ is an independent, strong $\{u,v\}$-critical for some vertices $u,v\in V(G)$. Since $S$ is an independent and strong critical set, $S' = S\cup \{u,v\}$ is also independent. Hence by (i), $S'$ is a MV set. But since $S$ is $\{u,v\}$-critical, the vertices $u$ and $v$ are not $S'$-visible, a contradiction. 

Assume second that $G$ contains no independent, strong critical set. Consider an arbitrary independent set $S$ of $G$. Let $u$ and $v$ be two vertices of $S$. Suppose for a contradiction that $u, v\in S$ are not $S$-visible. Set $T = \left(S\cap I_G[u,v]\right)\setminus \{u,v\}$. Since $u$ and $v$ are not $S$-visible, $T$ blocks all shortest $u,v$-paths, that is, $T$ is $\{u,v\}$-critical. In addition, since $S$ is independent, $T$ is a strong $\{u,v\}$-critical set. This contradiction completes the argument. 
\qed

Setting $I_k[u,v] = \{w\in I_G[u,v]:\ d_G(u,w) = k\}$, where $0\le k\le d_G(u,v)$, we have the following consequence of Theorem~\ref{thm:indep-is-MV}. 

\begin{corollary}
\label{cor:m-visibility and ind}
For a  graph $G$ with $\diam(G) = 4$, the following statements are equivalent.
\begin{enumerate}
\item [(i)] Every independent set of $G$ is a MV set.
\item[(ii)] If $d_G(u,v) = 4$, then $G[I_2[u,v]]$ contains an edge.    
\end{enumerate}
\end{corollary}

\proof
Since $\diam(G) = 4$, we infer that the condition (ii) of Theorem~\ref{thm:indep-is-MV} and the condition (ii) of Corollary~\ref{cor:m-visibility and ind} are equivalent. 
\qed

Corollary~\ref{cor:m-visibility and ind} does not extend (at least in a straightforward way) to graphs with diameter at least five. To see it, consider the sporadic example drawn in Fig.~\ref{fig:diam5}. Its diameter is five, the four black vertices form an independent set $S$, but the vertices $x$ and $y$ are not $S$-visible.

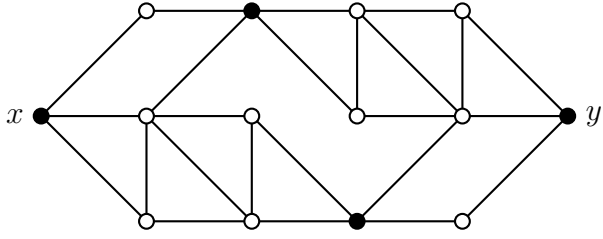
\begin{figure}[ht!]
\begin{center}
\begin{tikzpicture}[scale=0.7,style=thick,x=1cm,y=1cm]
\def\vr{4pt}
\def\radius{3cm}
    
\coordinate(x1) at (0,0);
\coordinate(x2) at (2,0);
\coordinate(x3) at (4,0);
\coordinate(x4) at (6,0);
\coordinate(x5) at (8,0);
\coordinate(x6) at (10,0);
\coordinate(y2) at (2,2);
\coordinate(y3) at (4,2);
\coordinate(y4) at (6,2);
\coordinate(y5) at (8,2);
\coordinate(z2) at (2,-2);
\coordinate(z3) at (4,-2);
\coordinate(z4) at (6,-2);
\coordinate(z5) at (8,-2);

\draw (x1) -- (y2) -- (y3) -- (y4) -- (y5) -- (x6) -- (z5) -- (z4) -- (z3) -- (z2) -- (x1);
\draw (x1) -- (x2) -- (x3) -- (z4) -- (x5) -- (x6);
\draw (z2) -- (x2) -- (z3) -- (x3);
\draw (x2) -- (y3) -- (x4) -- (y4) -- (x5) -- (y5);
\draw (x4) -- (x5);

\draw(x1)[fill=black] circle(\vr);
\draw(x2)[fill=white] circle(\vr);
\draw(x3)[fill=white] circle(\vr);
\draw(x4)[fill=white] circle(\vr);
\draw(x5)[fill=white] circle(\vr);
\draw(x6)[fill=black] circle(\vr);
\draw(y2)[fill=white] circle(\vr);
\draw(y3)[fill=black] circle(\vr);
\draw(y4)[fill=white] circle(\vr);
\draw(y5)[fill=white] circle(\vr);
\draw(z2)[fill=white] circle(\vr);
\draw(z3)[fill=white] circle(\vr);
\draw(z4)[fill=black] circle(\vr);
\draw(z5)[fill=white] circle(\vr);

\node at (-0.5,0) {$x$};
\node at (10.5,0) {$y$};

\end{tikzpicture}
\caption{An independent set of a diameter five graph which is not a MV set}
\label{fig:diam5}
\end{center}
\end{figure}

Using the strong product, we can construct graphs of arbitrarily large diameter in which every independent set is a (total) MV set. 

\begin{proposition}\label{prop:strong product}
If $n\ge 2$ and $G$ is a graph, then every independent set of $G\boxtimes K_n$ is a total MV set.  
\end{proposition}

\proof
Let $S$ be an independent set of $G\boxtimes K_n$ and let $x,y\in V(G)$. Let $P$ be an arbitrary shortest $x,y$-path. Assume that $P$ contains an inner vertex $(g,i)\in S$. If $i'\in V(K_n)$, $i'\ne i$, then $(g,i')\notin S$. In $P$, we can replace $(g,i)$ with $(g,i')$ to get another shortest $x,y$-path. Continuing this process as long as necessary we arrive at a shortest $x,y$-path which contains no inner vertices of $S$. 
\qed

\section{Algorithmic time complexity of related problems}
\label{sec:NP-c}

The independent mutual-visibility number of a graph $G$ was defined in~\cite{cicerone-2023} as the maximum cardinality of a MV set of $G$ that is also independent. Here we refer to it as \emph{independent MV number} and denote it by $\mui(G)$.
The \emph{independent total, dual, and outer MV numbers} are defined analogously and denoted by $\muti(G)$, $\mudi(G)$, and $\muoi(G)$, respectively. By definition, 
\begin{equation} \label{eq:mui-1}
    0 \leq \sigma^{\mbox i}(G) \leq \min\{\alpha(G), \sigma(G)\}
\end{equation}
 holds for every $\sigma \in \{\mu, \mut, \mud, \muo\}$ and graph $G$. It also follows by the definitions that
\begin{equation} \label{eq:mui-2}
  \muti(G) \leq \mudi(G) \leq \mui(G) \quad \mbox{and}
\quad \muti(G) \leq \muoi(G) \leq \mui(G).  
\end{equation} 

\begin{theorem} \label{thm:NP}
    For each parameter $\sigma \in \{\mui, \muti, \mudi, \muoi\}$, the following statements hold.
    \begin{itemize}
        \item[$(i)$] For every fixed $d \ge 4$, it is NP-complete to decide whether $\sigma(G) \ge k$ holds over the class of graphs of diameter $d$, where the integer $k$ is a part of the input. 
        \item[$(ii)$]  For every fixed $d \ge 4$, it is NP-hard to decide whether $\sigma(G) = \alpha(G)$ holds over the class of graphs of diameter $d$. 
    \end{itemize}
    Further, the NP-completeness of the problem in (i) and the NP-hardness of the problem in (ii) remain true for graphs of diameter three when  $\sigma \in \{\muti, \mudi, \muoi\}$.
\end{theorem}

\proof To prove (i) and (ii), we use the same polynomial-time reduction from the 3-SAT problem. Let $F= C_1 \wedge \cdots \wedge C_\ell$ be an instance of the 3-SAT problem over the $n$ variables $x_1, \dots, x_n$. Hence, for every $j \in [\ell]$, the clause $C_j$ is the disjunction of three literals over $\{x_1, \dots, x_n\}$. It can be assumed that every positive and negative literal is present in some clause. We construct a graph $G_d(F)$ of diameter $d$ for every $d \ge 4$, and prove that, for each $\sigma \in \{\mui, \muti, \mudi, \muoi\}$, the relations  $\sigma(G_d(F))\ge n+\ell + d-3$ and $ \sigma(G_d(F))= \alpha(G_d(F))$ hold if and only if $F$ is satisfiable.
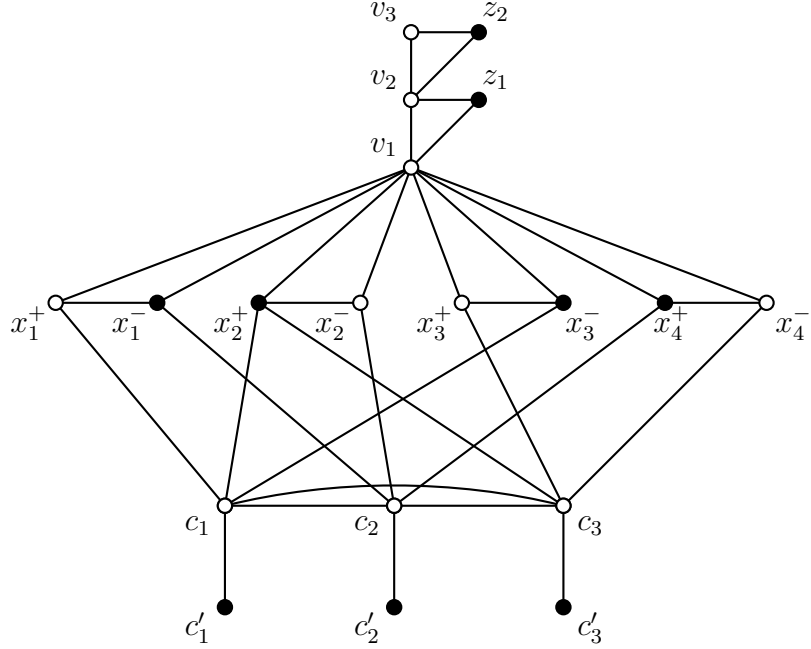
\begin{figure}[ht!]
\begin{center}
\begin{tikzpicture}[scale=0.9,style=thick,x=1cm,y=1cm]
\def\vr{3pt}
\begin{scope}[xshift=-0cm, yshift=0cm] 
\coordinate(x1+) at (0.0,0.0);
\coordinate(x1-) at (1.5,0.0);
\coordinate(x2+) at (3.0, 0.0);
\coordinate(x2-) at (4.5,0.0);
\coordinate(x3+) at (6.0,0.0);
\coordinate(x3-) at (7.5,0.0);
\coordinate(x4+) at (9.0,0.0);
\coordinate(x4-) at (10.5,0.0);
\coordinate(c1) at (2.5,-3.0);
\coordinate(c2) at (5.0,-3.0);
\coordinate(c3) at (7.5,-3.0);
\coordinate(c11) at (2.5,-4.5);
\coordinate(c21) at (5.0,-4.5);
\coordinate(c31) at (7.5,-4.5);
\coordinate(v1) at (5.25,2);
\coordinate(v2) at (5.25,3);
\coordinate(v3) at (5.25,4);
\coordinate(z1) at (6.25,3);
\coordinate(z2) at (6.25,4);
\draw (x1+) -- (x1-); 
\draw (x2+) -- (x2-); 
\draw (x3+) -- (x3-); 
\draw (x4+) -- (x4-); 
\draw (x1+) -- (v1);  
\draw (x1-) -- (v1); 
\draw (x2+) -- (v1);  
\draw (x2-) -- (v1);
\draw (x3+) -- (v1);  
\draw (x3-) -- (v1); 
\draw (x4+) -- (v1);  
\draw (x4-) -- (v1);
\draw (v1) -- (v2) -- (z1) -- (v1);
\draw (v2) -- (v3) -- (z2) -- (v2);
\draw (c1) -- (x1+);  
\draw (c1) -- (x2+);
\draw (c1) -- (x3-);
\draw (c2) -- (x1-); 
\draw (c2) -- (x2-);
\draw (c2) -- (x4+);
\draw (c3) -- (x2+); 
\draw (c3) -- (x3+);
\draw (c3) -- (x4-);
\draw (c1) -- (c11);
\draw (c2) -- (c21);
\draw (c3) -- (c31);
\draw (c1) -- (c2);
\draw (c2) -- (c3);

\draw(x1+)[fill=white] circle(\vr);
\draw(x1-)[fill=black] circle(\vr);
\draw(x2+)[fill=black] circle(\vr);
\draw(x2-)[fill=white] circle(\vr);
\draw(x3+)[fill=white] circle(\vr);
\draw(x3-)[fill=black] circle(\vr);
\draw(x4+)[fill=black] circle(\vr);
\draw(x4-)[fill=white] circle(\vr);
\draw(c1)[fill=white] circle(\vr);
\draw(c2)[fill=white] circle(\vr);
\draw(c3)[fill=white] circle(\vr);
\draw(c11)[fill=black] circle(\vr);
\draw(c21)[fill=black] circle(\vr);
\draw(c31)[fill=black] circle(\vr);
\draw(v1)[fill=white] circle(\vr);
\draw(v2)[fill=white] circle(\vr);
\draw(v3)[fill=white] circle(\vr);
\draw(z1)[fill=black] circle(\vr);
\draw(z2)[fill=black] circle(\vr);
\draw (2.5,-3.0) .. controls (4.1,-2.6) and (5.9,-2.6) .. (7.5,-3.0);
\node at (-0.4,-0.3) {$x_1^{+}$};
\node at (1.1,-0.3) {$x_1^{-}$};
\node at (2.6,-0.3) {$x_2^{+}$};
\node at (4.1,-0.3) {$x_2^{-}$};
\node at (5.6,-0.3) {$x_3^{+}$};
\node at (7.8,-0.3) {$x_3^{-}$};
\node at (9.1,-0.3) {$x_4^{+}$};
\node at (10.9,-0.3) {$x_4^{-}$};
\node at (4.85,2.3) {$v_1$};
\node at (4.85,3.3) {$v_2$};
\node at (4.85,4.3) {$v_3$};
\node at (6.50,3.3) {$z_1$};
\node at (6.50,4.3) {$z_2$};
\node at (2.1,-3.3) {$c_1$};
\node at (4.6,-3.3) {$c_2$};
\node at (7.9,-3.3) {$c_3$};
\node at (2.1,-4.8) {$c_1'$};
\node at (4.6,-4.8) {$c_2'$};
\node at (7.9,-4.8) {$c_3'$};
\draw(c1)[fill=white] circle(\vr);
\draw(c2)[fill=white] circle(\vr);
\draw(c3)[fill=white] circle(\vr);
\end{scope}
\end{tikzpicture}
\caption{Graph $G_5(F)$ for $F = (x_1 \lor x_2 \lor \overline{x}_3) \land (\overline{x}_1 \lor \overline{x}_2 \lor x_4) \land(x_2 \lor x_3 \lor \overline{x}_4)$. Black vertices form an independent total MV set of cardinality $\alpha(G_5(F))$.}
\label{fig:4}
\end{center}
\end{figure}
\paragraph{Construction of $G_d(F)$.} Define first $X=\{x_1^+, x_1^-, \dots, x_n^+, x_n^-\}$, where the elements of $X$ are vertices in $G_d(F)$ representing the positive and negative literals over $\{x_1, \dots, x_n\}$ in a natural way. Let $C=\{c_1, \dots, c_\ell\}$ be the set of clause vertices and $C'= \{ c_1', \dots, c_\ell'\}$ be the set of leaves attached to the clause vertices. We further define the vertex set $Z= \{v_1, \dots, v_{d-2}\} \cup \{z_1, \dots, z_{d-3}\}$. To obtain $G_d(F)$, we set $X \cup C \cup C' \cup Z$ as its vertex set, and add the edges $x_i^+x_i^-$, $x_i^+v_1$, $x_i^-v_1$ for every $i \in [n]$,  the edges $c_jc_j'$ for every $j \in [\ell]$. Further, we make $C$ a complete graph and insert the edges $v_sv_{s+1}$, $v_sz_s$, and $v_{s+1}z_s$ into $Z$ for every $s \in [d-3]$. Observe that the order of $G_d(F)$ is $1+2(d-3) + 2n +2\ell$ and $\diam(G_d(F))=d$. For an illustration, see Fig.~\ref{fig:4}.

\begin{claim} \label{cl:1}
 It holds that $\alpha(G_d(F))= n+\ell+d-3$. Further, if $S$ is a maximum independent set in $G_d(F)$, then $|S \cap Z|= d-3$, $|S \cap X|=n$, and either $S \cap C=\emptyset$ and $S \cap C'=C'$, or $S \cap C=\{c_q\}$ and $S \cap C' = C' \setminus \{c_q'\}$ for an index $q \in [\ell]$. 
\end{claim}
\noindent\emph{Proof.} Let $S$ be an independent set in $G_d(F)$. Since $Z$ can be covered by $d-3$ triangles, $|S \cap Z|\leq d-3$. Each of $X$ and $C \cup C'$ admits a perfect matching and therefore, $|S \cap X|\leq n$ and $|S \cap (C \cup C')| \leq \ell$.
This implies $\alpha(G_d(F)) \leq n+\ell+d-3$. Since 
$$S=\{z_1, \dots, z_{d-3}, x_1^+, \dots, x_n^+, c_1', \dots, c_\ell'\}$$
 is an independent set of cardinality $n+\ell+d-3$, the equality follows. We further recall that $C$ induces a complete graph and hence $|S \cap C| \leq 1$ holds for any independent set $S$.
 \smallqed
 
 Note that Claim~\ref{cl:1} and the upper bound in~\eqref{eq:mui-1} imply that the relations $\sigma(G_d(F))\ge  n+\ell+d-3$, $\sigma(G_d(F))=  n+\ell+d-3$, and $\sigma(G_d(F))=\alpha(G_d(F))$ are equivalent for every $\sigma \in \{\mui, \muti, \mudi, \muoi\}$.
\begin{claim} \label{cl:2}
  If $F$ is satisfiable, then $\sigma(G_d(F))=  n+\ell+d-3$ for each $\sigma \in \{\mui, \muti, \mudi, \muoi\}$. 
\end{claim}
\noindent\emph{Proof.}  Let $\phi\colon \{x_1, \dots, x_n\} \rightarrow \{\mbox{true, false}\}$ a truth assignment that satisfies $F$ and define 
$$ S'=\{z_1, \dots, z_{d-3}, c_1', \dots, c_\ell'\} \cup \{x_i^+\colon \phi(x_i)=\mbox{false}\} \cup 
\{x_i^-\colon \phi(x_i)=\mbox{true}\}.
$$
We first prove that $S'$ is an independent total MV set. It is straightforward to check that $S'$ is an independent set of cardinality $n+\ell+d-3$. To see that it is a total MV set, by Theorem~\ref{thm:distance two}, it suffices to show that any two vertices at a distance of two apart are $S'$-visible. Recalling that any two nonadjacent vertices from $X$ can see each other via $v_1$, the $S'$-visibility is easy to check inside $Z \cup X$.
As for $C \cup C'$, if two vertices from $C \cup C'$ are at a distance of two, then one of them is a vertex $c_j$ from $C$, and the other is a vertex $c_p'$ from $C'$ such that $j\neq p$. In this case, the vertices are $S'$-visible via $c_p$. 

What remains to check is the case of two vertices $u \in Z \cup X$ and $w \in C \cup C'$ with $d(u,w)=2$. If $u \in X$ and $w \in C \cup C'$, they can always see each other via a vertex from $C$. If $u \in Z$ and $w \in C$, then $u=v_1$ and $w=c_j$ for an integer $j \in [\ell]$ . Since the clause $C_j$ is satisfied by a literal $x_i$ or $\overline{x_i}$, the corresponding literal vertex, namely $x_i^+$ or $x_i^-$, does not belong to $S'$, and it therefore ensures the $S'$-visibility of $v_1$ and $c_j$. Consequently, $S'$ is an independent total MV set and $\muti(G_d(F))= n+\ell+d-3= \alpha(G_d(F))$.  Then, by definition, $S'$ is also an MV set, a dual MV set, and an outer MV set. Thus Claim~\ref{cl:2} also holds for $\sigma \in \{\mui, \mudi, \muoi\}$.  \smallqed 
\begin{claim}  \label{cl:3}
  For each invariant $\sigma \in \{\mui, \muti, \mudi, \muoi\}$, the equality  $\sigma(G_d(F))= n+\ell+d-3$ implies that $F$ is satisfiable. 
\end{claim} 
\noindent\emph{Proof.} Let $S''$ be an independent MV set of cardinality $n+\ell+d-3$ in $G_d(F)$. Claim~\ref{cl:1} then implies $|S'' \cap Z|= d-3$, $|S'' \cap X|=n$, and $|S'' \cap (C \cup C')|=\ell.$ Consequently, $S''$ contains exactly one vertex from each pair $\{x_i^+, x_i^-\}$, and $v_1 \notin S''$. We may assume, without loss of generality, that $S'' \cap Z=\{z_1, \dots, z_{d-3}\}$. Consider the truth assignment 
\begin{equation}  \label{eq:mui-3}
    \phi(x_i)= \left\{ \begin{array}{ll}
		\mbox{true}; & x_i^- \in S'', \\		
		\mbox{false}; & x_i^+ \in S''.
	\end{array} \right. 
\end{equation}

Assume first that $S''\cap C'=C'$. For each $j \in [\ell]$, the $S''$-visibility of $z_1 \in S''$ and $c_j' \in S''$ implies that there exists a path $z_1v_1uc_jc_j'$  such that $u \notin S''$. If $u=x_i^+$ for an index $i \in [n]$, then the clause $C_j$ in $F$ contains the literal $x_i$ and, according to~\eqref{eq:mui-3}, the clause $C_j$ is satisfied. The argument is analogous for $u=x_i^-$. It implies the satisfiability of $F$. 

In the second case, $S \cap C=\{c_q\}$ and $S \cap C' = C' \setminus \{c_q'\}$ for an index $q \in [\ell]$. As for the first case, if $c_j' \in S''$, then the clause $C_j$ is satisfied by the truth assignment in~\eqref{eq:mui-3}. For the vertex $c_q$, as $S''$ is an independent set, no neighbor of $c_q$ belongs to $S''$ and the clause $C_j$ is satisfied by $\phi$. Again, it shows that $F$ is satisfiable. 

If $\sigma \in \{\muti, \mudi, \muoi\}$, then 
$\sigma(G_d(F))=n+\ell+d-3$ implies $\mui(G_d(F))=n+\ell+d-3$, and the satisfiability of $F$ follows from the proof above.
\smallqed
\medskip

From a 3-SAT instance $F$, the graph $G_d(F)$ can be obtained in polynomial time for every $d\ge 4$. Let $\sigma \in \{\mui,\mudi, \muoi, \muti\}$. Claims~\ref{cl:1}, \ref{cl:2}, and \ref{cl:3} (together with~\eqref{eq:mui-1} and \eqref{eq:mui-2}) show that each of $\sigma(G_d(F))\ge n + \ell+d-3$ and $\sigma(G_d(F))= \alpha(G_d(F))$ holds if and only if $F$ is satisfiable. 
Therefore, the 3-SAT problem can be reduced to both decision problems of whether $\sigma(G) \ge k$ and $\sigma(G)=\alpha(G)$ hold for a graph of diameter $d$ for each fixed integer $d \ge 4$. This proves the NP-hardness of the decision problems in (i) and (ii).

Finally, we note that while the problems in (i) clearly belong to NP, we cannot state the same for the problems in (ii). This finishes the proof for (i) and (ii).
\paragraph{Graphs of diameter three.}
For an instance $F$ of the 3-SAT problem, we obtain the graph $G_3(F)$ by deleting $z_1$ and $v_2$ from $G_4(F)$. Then $\diam(G_3(F))=3$ and $\alpha(G_3(F))=n + \ell$. Assuming that $F$ is satisfiable, we can construct the set $S'$ as in the proof of Claim~\ref{cl:2} and prove that $S'$ is an independent total (and consequently dual and outer) MV set of cardinality $n + \ell$. For the other direction, we assume that there is an independent outer/dual/total MV set $S''$ of cardinality $n + \ell$ in $G_3(F)$. For the outer and total mutual-visibility, we consider the visibility of $v_1$ and the leaves $c_j' \in S''$, for $j \in [\ell]$, to conclude that every clause is satisfied by the truth assignment given in~\eqref{eq:mui-3}; while for the dual mutual-visibility we have the same conclusion by considering $v_1$ and the clause vertices $c_j$. This proves that the problems in (i) and (ii) remain NP-hard for graphs of diameter three when $\sigma \in \{\muti, \mudi, \muoi\}$.  
\qed

It is known \cite[Lemma 2.1]{cicerone-2023} that $\diam(G) \leq 3$ implies $\mui(G) = \alpha(G)$. Therefore, the NP-hardness of the decision problem in Theorem~\ref{thm:NP}(ii) cannot be extended to graphs of diameter three if the invariant $\mui(G)$ is considered.  
On the other hand, it is known by a simple construction~\cite{garey} that the decision problem of $\alpha(G) \ge k$ remains NP-complete over the class of graphs of diameter two and three. We may then conclude the following result.

\begin{proposition} \label{prop:NP}
    For $d \in \{2,3\}$, it is NP-complete to decide whether $\mui(G) \ge k$ holds over the class of graphs of diameter $d$, where $k$ is part of the input. 
\end{proposition}

\section{Open problems}
\label{sec:problems}

Proposition~\ref{prop:outer-inde in diam 2} gives a partial solution to the following problem. 

\begin{problem}
\label{prob:indep-outer}
Characterize the graphs in which every independent set is an outer MV set. 
\end{problem}

Comparing Theorems~\ref{thm:dense-diamond-independent} and \ref{thm:total-only-independent} we can derive that the following statements are equivalent for every connected graph $G$. 
\begin{enumerate}
\item[(i)] For every $X \subseteq V(G)$, the set $X$ is a  total MV set of $G$ if and only if $X$ is independent.
\item[(ii)] $G$ is a dense diamond graph and every edge of $G$ is the middle edge of a convex diamond.
\end{enumerate}

We were unable to find any graph that satisfies (ii) above, hence we pose the following:

\begin{problem}
Does there exist a dense diamond graph such that every edge of $G$ is the middle edge of a convex diamond?    
\end{problem}

Related to mutual-visibility and independence, we pose:

\begin{problem}
Characterize the graphs $G$ and $H$ for which every independent set of $G\boxtimes H$ is a MV set.
\end{problem}

\section*{Acknowledgments}

This work was supported by the Slovenian Research Agency (ARIS) under the grants P1-0297, N1-0285, N1-0355, N1-0431, J1-70045.

\end{document}